\documentclass[12pt]{amsart}

\usepackage{amssymb,amsthm,amsmath}
\usepackage{hyperref}
\usepackage[alphabetic]{amsrefs}

\usepackage[utf8]{inputenc}
\usepackage[T1]{fontenc}
\usepackage{lmodern}

\usepackage{enumitem}
\setlist[enumerate]{label=\rm{(\roman*)}, leftmargin=*, widest=iii}

\usepackage[margin=3cm]{geometry}

\newtheorem{theorem}{Theorem}[section]
\newtheorem{corollary}[theorem]{Corollary}
\newtheorem{lemma}[theorem]{Lemma}
\newtheorem{proposition}[theorem]{Proposition}

\theoremstyle{remark}
\newtheorem{remark}[theorem]{Remark}
\newtheorem{question}[theorem]{Question}
\newtheorem{example}[theorem]{Example}

\numberwithin{equation}{section}

\newcommand{\N}{\mathbb{N}}

\newcommand{\R}{\mathbb{R}}
\newcommand{\C}{\mathbb{C}}

\newcommand{\M}{\mathrm{Mix}}
\newcommand{\MM}{\overline{\mathrm{Mix}}}
\newcommand{\sa}{\mathrm{sa}}

\renewcommand{\epsilon}{\varepsilon}
\renewcommand{\phi}{\varphi}
\renewcommand{\leq}{\leqslant}
\renewcommand{\geq}{\geqslant}

\date{}
\title{Maximally unitarily mixed states on a C*-algebra}

\author
[Archbold]{Robert Archbold}
\address{Robert Archbold
	\\Institute of Mathematics
	\\University of Aberdeen
	\\King's College
	\\Aberdeen AB24 3UE
	\\Scotland
	\\United Kingdom
} \email{r.archbold@abdn.ac.uk}

\author
[Robert]{Leonel Robert}
\address{Leonel Robert
	\\Department of Mathematics
	\\University of Louisiana at Lafayette
	\\Lafayette, 70504-3568
	\\USA
} \email{lrobert@louisiana.edu}

\author
[Tikuisis]{Aaron Tikuisis}
\address{Aaron Tikuisis
	\\Department of Mathematics and Statistics
	\\University of Ottawa
	\\585 King Edward
	\\Ottawa, ON K1N 6N5
	\\Canada
}\email{aaron.tikuisis@uottawa.ca}

\thanks{A.T.\ was partially supported by an NSERC Postdoctoral Fellowship and through the EPSRC grant EP/N00874X/1.}

\begin{document}
	
\begin{abstract}
We investigate the set of maximally mixed states of a C*-algebra, extending previous work by Alberti on von Neumann algebras.
We show that, unlike for von Neumann algebras, the set of maximally mixed states of a C*-algebra may fail to be  weak* closed.
We obtain, however,  a concrete description of the weak* closure of this set, in terms of tracial states and states which factor through simple traceless quotients.  For C*-algebras with the Dixmier property or with Hausdorff primitive spectrum
we are able to advance our investigations further. In the latter case we obtain a concrete description of the set of maximally mixed states in terms
of traces and extensions of the states of a closed two-sided ideal.  We pose several questions.
\end{abstract}	
	
\subjclass[2010]{46L30}
\keywords{States of C*-algebras, Unitary mixings, Dixmier property}	

\maketitle	

\section{Introduction}
Investigations into the entropy and irreversibility of the states of a physical system lead  to the consideration of the ``more mixed than'' or ``more chaotic than'' pre-order
on the space of states of a C*-algebra.  This pre-order, first introduced by Uhlmann,  has been investigated for the state spaces of matrix algebras and, more generally,
of von Neumann algebras,  by Alberti, Uhlmann, Wehrl, and others  (\cite{Alb-Uhl0,Alb-Uhl, Uhlmann, Wehrl}).  Uhlmann also introduced a distinguished collection of states:  the maximally mixed states. These are the  maximal elements in the ``more mixed than'' pre-order.  In \cite[Theorem 5.2]{Alberti},
Alberti gave a complete description of the maximally mixed states of a von Neumann algebra.  In this paper we undertake the study of the maximally mixed states of a C*-algebra. In particular, we probe the extent to which Alberti's theorem can be extended to arbitrary  C*-algebras.

Let us be more specific. Let $A$ be a C*-algebra.  Given two states $\phi$ and $\psi$ on $A$, let's say that $\psi$ is more (unitarily) mixed than $\phi$ if $\psi$ belongs to the weak* closure of the convex hull of the unitary conjugates of $\phi$. A state $\phi$ is maximally (unitarily) mixed if whenever $\psi$ is more mixed than $\phi$ then $\phi$ is also more mixed than $\psi$. Maximally mixed states are guaranteed to exist by weak* compactness and Zorn's lemma (in fact, given any state $\phi$ there exist maximally mixed states that are more mixed than $\phi$).
We denote the set of maximally mixed states of $A$ by $S_\infty(A)$.

The main question that we address here is ``can the set $S_\infty(A)$ be described more concretely?''.  The tracial states on $A$ are obviously maximally mixed.
Another source of maximally mixed states on $A$ is the quotients $A/M$ that are simple and have no bounded traces.
The states (on $A$) that factor through these quotients are  also maximally mixed.
Alberti showed that if  $A$ is a von Neumann algebra  then $S_\infty(A)$ is the weak* closure of the  convex hull of the tracial states and the states that factor through simple traceless quotients (see \cite{Alberti}, though it is not quite stated this way).
We demonstrate below with natural examples  that  the set $S_\infty(A)$ need not always be weak* closed.
It is the case, however, that the weak* closure of $S_\infty(A)$ is precisely
the weak* closure of the  convex hull of the tracial states and the states factoring through simple traceless quotients (Theorem \ref{maxmixS}).
We leave open the question of the convexity of $S_\infty(A)$.

For C*-algebras with the Dixmier property we are able to advance our understanding of $S_\infty(A)$ further.
Recall that $A$ is said to have the Dixmier property if for every   $a\in  A$ the norm closure of the convex hull of the unitary conjugates of $a$ intersects the center of $A$.
Von Neumann algebras have the Dixmier property (by Dixmier's approximation theorem), but the class of C*-algebras with the Dixmier property is much larger (see \cite{ART}). We show  that if $A$ has the Dixmier property then $S_\infty(A)$
is convex and weakly closed.  Further, we obtain a necessary and sufficient condition (involving the  primitive spectrum) for  $S_\infty(A)$ to be weak* closed (Theorem \ref{DPweaklyclosed}).

 The paper is organized as follows: In Section \ref{Prelims} we introduce notation. In Section \ref{Maxmixed} we embark on the investigation of the maximally mixed states of an arbitrary C*-algebra. In Section \ref{DPcase} we consider C*-algebras with the Dixmier property. In Section \ref{Hausdorffcase} we rely on the results from the previous sections to obtain a concrete description of  the set of maximally mixed states of a C*-algebra with Hausdorff primitive spectrum.

\section{Preliminaries on Dixmier sets}	\label{Prelims}
Let $A$ be a C*-algebra.  We denote by $A_\sa$ the set of self-adjoint elements of $A$ and by $A_+$ the set of  positive elements of $A$.
If $A$ is unital we denote by $U(A)$ the group of unitary elements of $A$.

We denote by $A^\sim$ the minimal  unitization of $A$, i.e., $A$ itself if $A$ is unital and the unitization $A+\C1$ if $A$ is non-unital.

Let $A^*$ denote the dual of $A$.  We denote by $A_\sa^*$ the set of self-adjoint functionals in $A^*$ and by $A_+^*$ the set  of positive functionals
in $A^*$.

We denote the convex hull of a set $S$ (in an affine space) by $\mathrm{co}(S)$.

\subsection{Dixmier sets in $A$ and $A^*$}
We call a set $C\subseteq A$ a Dixmier set if it is convex, norm-closed, and invariant under unitary conjugation. The latter means that $uCu^*\subseteq C$ for all unitaries $u\in U(A^\sim)$.  We will largely work with singly generated Dixmier sets. Given $a\in A$ we denote by $D_A(a)$ the smallest Dixmier set containing $a$.

We let $A$, and more generally $M(A)$ (the multiplier algebra of $A$), act on $A^*$ in the usual way:
if $a\in M(A)$  and $\phi\in A^*$ then
\[
a\phi(x):=\phi(ax),\quad(\phi a)(x):=\phi(xa) \qquad (x\in A).
\]
A set $C\subseteq A^{*}$ is called a Dixmier set if it is convex, weak* closed, and invariant under unitary conjugation. The latter condition means that $uCu^*\subseteq C$ for all unitaries $u\in  U(A^\sim)$. Given $\phi\in A^*$ we denote by $D_A(\phi)$ the Dixmier set generated by $\phi$, i.e., the smallest Dixmier set containing $\phi$. Since $D_A(\phi)$ is weak* closed and bounded, it is weak* compact.

We shall make frequent use of the standard fact that $A$ is the dual of $A^*$ when the latter is  endowed with the weak* topology. This, combined with the Hahn-Banach theorem, implies that elements of $A$ separate disjoint weak* compact convex sets in $A^*$.

Let $\mathcal V$ be a subgroup of the unitary group $U(M(A))$ of $M(A)$.
On some occasions we will need  more general versions of the sets defined above
where the unitaries range through $\mathcal V$ rather than $U(A^\sim)$. Thus, given $a\in A$ we define
$D_A(a,\mathcal V)$ as the smallest norm-closed  convex subset of $A$ containing $a$ and invariant under conjugation by unitaries in
$\mathcal V$.  Similarly, given $\phi\in A^*$ we define $D_A(\phi,\mathcal V)$
as the the smallest weak* closed  convex subset of $A^*$ containing $\phi$ and invariant under conjugation by unitaries in
$\mathcal V$.

\subsection{Mixing operators}
Let $\mathcal V$ be a subgroup of the unitary group $U(M(A))$ of $M(A)$.
We call a linear operator $T\colon A\to A$   a $\mathcal V$-mixing operator if it is
defined by an equation of the form
\[
Ta = \sum^n_{j=1}\lambda_ju_ja u^*_j  \qquad (a\in A),
\]
where $n\in \N$, $\lambda_j>0$, $u_j\in \mathcal V$ ($1\leq j\leq n$), and $\sum^n_{j=1}\lambda_j=1$. Elementary properties of such operators are described in \cite[2.2]{Arch-PLMS}.  We denote
by $\M(A,\mathcal V)$ the set of $\mathcal V$-mixing operators on $A$. If $\mathcal V=U(A^\sim)$
we simply write $\M(A)$.
Notice that
\[
D_A(a,\mathcal V)=\overline{\{Ta : T\in \M(A,\mathcal V)\}}^{\|\cdot\|}.
\]

We also call an operator $T\colon A^*\to A^*$ a $\mathcal V$-mixing operator if it is the adjoint of a $\mathcal V$-mixing operator on $A$.
In this case $T$ has the form
\[
T\phi = \sum^n_{j=1}\lambda_ju_j\phi u^*_j  \qquad (\phi\in A^*),
\]
where $n\in \N$, $\lambda_j>0$, $u_j\in \mathcal V$ ($1\leq j\leq n$), and $\sum^n_{j=1}\lambda_j=1$.
Observe  that $T$ is positive ($T\phi\geq 0$ for all $\phi\geq 0$) and contractive.
We denote the set of $\mathcal V$-mixing operators on $A^*$ by $\M(A^*,\mathcal V)$ or simply by $\M(A^*)$ if $\mathcal V=U(A^\sim)$. Notice  that
\[
D_A(\phi,\mathcal V)=\overline{\{T\phi: T\in \M(A^*,\mathcal V)\}}^{\mathrm{weak}^*}.
\]

\begin{lemma}\label{DphiphiD}
	Let $a\in A$ and $\phi\in A^*$. Then
	\begin{equation}
	D_A(\phi,\mathcal V)(a)=\overline{\phi(D_A(a,\mathcal V))}.
	\end{equation}
\end{lemma}
\begin{proof}
	Since $D_A(\phi,\mathcal V)$ is weak$^*$-compact, $D_A(\phi,\mathcal V)(a)$ is a closed	subset of $\C$. To prove the lemma it suffices to show that $\phi(D_A(a,\mathcal V))$ is a dense subset of $D_A(\phi,\mathcal V)(a)$.  Let $T\in \M(A,\mathcal V)$. Then
	$(T^*\phi)(a)=\phi(Ta)$.
	Letting $T$ range through all $\M(A,\mathcal V)$ the left side is dense in $D_A(\phi,\mathcal V)(a)$ while the right side is dense in $\phi( D_A(a,\mathcal V))$.
\end{proof}

We will find it convenient to work with more general unitary mixing  operators on $A^*$. We let $\MM(A^*,\mathcal V)$ denote the closure of $\M(A^*,\mathcal V)$
in the point-weak$^*$ topology on $B(A^*)$ (the bounded linear operators on $A^*$). If $\mathcal V=U(A^\sim)$ we simply write $\MM(A^*)$.
Since a limit in the point-weak$^*$ topology of positive contractions is again a positive contraction,  all $T\in \MM(A^*,\mathcal V)$ are positive contractions.
Since the unit ball of   $B(A^*)$ is compact in the point-weak$^*$ topology, $\MM(A^*,\mathcal V)$ is  a compact set in this topology.

\begin{lemma}\label{MMoperators}
	Let $\phi\in A^*$. Then $D_A(\phi,\mathcal V)=\{T\phi :T\in \MM(A^*,\mathcal V)\}$.
\end{lemma}	
\begin{proof}
	Clearly, $T\phi\in D_A(\phi,\mathcal V)$ for all $T\in \MM(A^*,\mathcal V)$. Suppose that  $ \psi\in D_A(\phi,\mathcal V)$. Then
	$T_i\phi\to \psi$ in the weak* topology for some net of $\mathcal V$-mixing operators $(T_i)_i$ on $A^*$.
	Passing to a  subnet of $(T_i)_i$ convergent in the point-weak* topology we get that $\psi=T\phi$ for some $T\in \MM(A^*,\mathcal V)$.
\end{proof}

\section{Maximally mixed functionals}\label{Maxmixed}
Let $\phi\in A^*$. If $\psi\in D_A(\phi)$ we say that $\psi$ is more unitarily mixed than $\phi$.
We say that $\phi $ is \emph{maximally (unitarily) mixed} if  $D_A(\phi)$ is minimal  with respect to the order by inclusion
in the lattice of weak$^*$-compact Dixmier subsets of $A^*$. Thus $\phi$ is maximally mixed if and only if for all $\psi\in D_A(\phi)$ we have $D_A(\psi)=D_A(\phi)$.

It follows from Zorn's lemma that any weak* compact Dixmier set contains a maximally mixed functional. In particular, $D_A(\phi)$ contains a maximally mixed functional for all $\phi\in A^*$.
Note also that (i) the zero functional is maximally mixed, (ii) if $\phi$ is tracial then $D_A(\phi)=\{\phi\}$ and hence $\phi$ is maximally mixed, and (iii) if $\phi$ is maximally mixed and $\lambda\in \C$
then $\lambda\phi$ is maximally mixed.

\begin{theorem}\label{jordan}
	Let $A$ be a C*-algebra and let $\phi\in A^*$ be maximally mixed. Then the  self-adjoint and skew-adjoint parts of $\phi$ are maximally mixed.
	If $\phi$ is self-adjoint, then its positive and negative parts are maximally mixed.
\end{theorem}	
\begin{proof}
Let $\phi_\sa$ denote the self-adjoint part of $\phi$. Let $\psi\in D_A(\phi_\sa)$. Then $\psi=T\phi_\sa$
for some $T\in \MM(A^*)$ (Lemma \ref{MMoperators}).  Mixing operators
in $\MM(A^*)$ preserve the self-adjoint part. So $\psi$ is the self-adjoint part of $T\phi$.
 Since $\phi$ is maximally mixed and $T\phi\in D_A(\phi)$, there exists $S\in \MM(A^*)$ such that $ST\phi=\phi$.  Taking self-adjoint  parts we get $S\psi=\phi_{\sa}$. Thus, $\phi_{\sa}\in D_A(\psi)$, as desired. The same argument applies to the skew-adjoint part.
	
	Suppose now that  $\phi$ is self-adjoint (and maximally mixed). Let us show first that $(T\phi)_+=T\phi_+$
	and $(T\phi)_-=T\phi_-$ for any $T\in \MM(A^*)$. Observe that $\|\psi\|\leq \|\phi\|$ for all $\psi\in D_A(\phi)$. But, since $\phi$ is maximally mixed, we  must have that $\|\psi\|=\|\phi\|$ for all $\psi\in D_A(\phi)$. That is,  all the functionals in $D_A(\phi)$ have the same norm. Applying $T$ on both sides of $\phi=\phi_+-\phi_-$
	we get $T\phi=T\phi_+-T\phi_-$. Then,
	\[
	\|T\phi_+\|+\|T\phi_-\|\leq \|\phi_+\|+\|\phi_-\|=\|\phi\|=\|T\phi\|.
	\]
	It follows that $T\phi_+$ and $T\phi_-$ are orthogonal (\cite[Lemma 3.2.3]{Ped}). By the uniqueness of the Jordan decomposition (\cite[Theorem 3.2.5]{Ped}), $(T\phi)_+=T\phi_+$
	and $(T\phi)_-=T\phi_-$.
	
	That $\phi_+$
	and $\phi_-$ are maximally mixed is now straightforward. For suppose  that $\psi\in D_A(\phi_+)$. By Lemma \ref{MMoperators}, there exists  $T\in \MM(A^*)$ such that   $\psi=T\phi_+$.  Further,
	since $\phi$ is maximally mixed,
	there exists $S\in \MM(A^*)$ such that   $ST\phi=\phi$. Then
	$S\psi=ST\phi_+=(ST\phi)_+=\phi_+$.  Thus, $\phi_+$ is maximally mixed. The same argument shows that $\phi_-$ is maximally mixed.
\end{proof}	

Due in part to the previous theorem, in the sequel our focus  will be on the positive maximally mixed functionals.
We warn however that it is not true that a self-adjoint functional whose positive and negative parts are maximally mixed is itself maximally mixed: see Example \ref{JordanConverseFails}.

\begin{theorem}\label{normclosed}
	Let $A$ be a C*-algebra. The set of maximally mixed functionals is a norm-closed subset of $A^*$.
\end{theorem}	
\begin{proof}
	Let $\phi\in A^*$ be in the norm-closure of the set of maximally mixed functionals.	Let $\psi\in D_A(\phi)$. By Lemma \ref{MMoperators},   there exists $T\in \MM(A^*)$ such that  $\psi=T\phi$.
	Let $\epsilon>0$.
	Then there exists a maximally mixed $\tilde\phi$ such that $\|\phi- \tilde\phi\|<\epsilon$.
	Since $T$ is a contraction,
	\[
	\|\psi-T\tilde\phi\|=\|T\phi-T\tilde\phi\|\leq \|\phi- \tilde\phi\|<\epsilon.
	\]
	Since $\tilde \phi$ is maximally mixed, there exists $S\in \MM(A^*)$
	such that $ST\tilde \phi=\tilde\phi$. Then,
	\[
	\|S\psi-\tilde\phi\|=\|S\psi-ST\tilde\phi\|\leq \|\psi-T\tilde\phi\|<\epsilon.
	\]
	So $\|\phi-S\psi\|< 2\epsilon$.  Since $D_A(\psi)$ is norm-closed, we
	have $\phi\in D_A(\psi)$ and hence $D_A(\psi)=D_A(\phi)$. Thus, $\phi$ is maximally mixed.
\end{proof}	

We will show in Examples \ref{notweaklyclosed} and \ref{notweaklyclosed2} that the set of maximally mixed functionals is not always weak* closed. We do have the following:

\begin{proposition} Let $A$ be a  unital C*-algebra and let $\phi\in A^*_+$.
	\begin{enumerate}
		\item
		Suppose that for every $a\in A_{\mathrm{sa}}$ and $\epsilon>0$
		there exists a maximally mixed $\phi'\in A_+^*$ such that $\phi'\leq \phi$ and
		$|\phi(a)-\phi'(a)|<\epsilon$. Then $\phi$ is maximally mixed.
		
		\item
		Suppose that for every $a\in A_{\mathrm{sa}}$ and $\epsilon>0$
		there exists a maximally mixed $\phi'\in A_+^*$ such that $\phi'\geq \phi$ and
		$|\phi(a)-\phi'(a)|<\epsilon$. Then $\phi$ is maximally mixed.
		
		\item
		Suppose that $(\phi_i)_i$ is a norm-bounded net of maximally mixed  functionals in $A^*_{+}$ which is either upward directed or downward directed relative to the order in $A_{+}^*$. Then the net is convergent and the limit is maximally mixed.
	\end{enumerate}		
\end{proposition}	
\begin{proof}
	(i) Let $\psi\in D_A(\phi)$ and suppose that $\psi=T\phi$, where  $T\in \MM(A^*)$. Suppose, towards a contradiction, that $\phi\notin D_A(\psi)$. Then by the Hahn-Banach theorem there exist $a\in A_{\sa}$, $t\in \R$ and $\epsilon>0$ such that $\rho(a)\leq t$ for all $\rho\in D_A(\psi)$ but $\phi(a)\geq t+\epsilon$. Replacing $a$ by $a+\|a\|1$ and $t$ by $t+\|a\|\|\phi\|$, we may assume that $a\geq0$.
	
	By hypothesis, there exists a maximally mixed functional $\phi'\in A_+^*$  such that $\phi'\leq \phi$ and $\phi'(a)\geq t+\epsilon/2$. Let $\psi'=T\phi'$.
	Note that, since $T$ is positive, $\psi'\leq \psi$.
	Since $\phi'$ is maximally mixed, $\phi'\in D_A(\psi')$.  Thus, there exists $S\in \MM(A^*)$ such that $S\psi'=\phi'$. Let   $\rho=S\psi$.
	Then $\phi'\leq \rho$ and so $\phi'(a)\leq \rho(a)\leq t$ since $a\geq0$. This contradicts the fact that $\phi'(a)\geq t+\epsilon/2$.	Thus $\phi\in D_A(\psi)$
	and hence $D_A(\psi)=D_A(\phi)$.

	(ii) This is similar to (i).

	(iii)  The convergence of the net follows from weak$^*$-compactness, monotonicity and the fact that $A$ is the linear span of $A_+$. The limit is maximally mixed by (i) and (ii).
\end{proof}

Next we prepare to examine the relation of the maximally mixed functionals of $A$ with those of its ideals and quotients.
Theorem \ref{MMextensions} will tell us that, given an ideal $J$ of $A$, maximal mixedness of a functional can be read off by its decomposition with respect to $A/J$ and $J$.
Part (i) of the following proposition is a classical key result used to prove permanence of the Dixmier property under suitable extensions; we use part (ii) in an analogous way to handle Dixmier sets of functionals.

\begin{proposition}\label{eqn: exp-unitary for A}
	Let $A$ be a C*-algebra, let $a\in A$ and let $\phi\in A^*$. The following are true:
	
	(i) $D_A(a)$ is equal to the  norm-closure of $\mathrm{co}\{e^{ih}ae^{-ih}: h\in A_{\sa}\}$.

	(ii) $D_A(\phi)$ is equal to the weak* closure  of $\mathrm{co}\{e^{ih}\phi e^{-ih}: h\in A_{\sa}\}$.
\end{proposition}	
\begin{proof}
	(i) For unital $A$, the result is given in \cite[Proposition 2.4]{Arch-PLMS}. For non-unital $A$, we apply this result to $A^\sim$ and use the fact that if $h\in A_{\sa}$ and $t\in \R$ then $e^{i(h+t1)}=e^{it}e^{ih}$.
	
	(ii) This follows from (i) and the Hahn-Banach theorem. Indeed, if (ii) fails to hold then there is a unitary conjugate of $\phi$ which does not belong to the weak* closure  of $\text{co}\{e^{ih}\phi e^{-ih}: h\in A_{\sa}\}$. Since $A^*$ with the weak$^*$-topology has dual space $A$,  it follows  by the Hahn-Banach separation theorem that there exists $u\in U(A^\sim)$, $a\in A$ and $t\in \R$ such that $\mathrm{Re}(\phi(uau^*))>t$ and $\mathrm{Re}(\phi(e^{ih}ae^{-ih}))\leq t$ for all $h\in A_{\sa}$. It follows from the last inequality and part (i) that $\mathrm{Re}(\phi(x))\leq t$ for all $x\in D_A(a)$. This contradicts the fact that $\mathrm{Re}(\phi(uau^*))>t$.
\end{proof}

\begin{proposition}\label{Dextensions}
	Let $J$ be a proper, closed two-sided  ideal of a  unital C*-algebra $A$. 	Let $\iota_J\colon J \to A$ and $q_J\colon A\to A/J$ denote the inclusion and  quotient maps.
	\begin{enumerate}
		\item	
		The adjoint map $\iota_J^*\colon A^*\to J^*$ maps  $D_A(\phi)$ onto
		$D_J(\phi|_J)$ for all $\phi\in A^*_+$.
		
		\item
We have
$D_A(\phi)=D_A(\phi,U(J+\C1))$	 for all  $\phi\in A^*_+$ such that $\|\phi\|=\|\phi|_J\|$.
		
		\item
		The adjoint map $q_J^*\colon (A/J)^*\to A^*$ maps
		$D_{A/J}(\phi)$ bijectively to $D_A(\phi\circ q_J)$ for all $\phi\in (A/J)^*_+$.
	\end{enumerate}
	
\end{proposition}	
\begin{proof}
If the ideal $J$ is a unital C*-algebra then $A\cong J\oplus A/J$ and all three results (i)-(iii) have a straightforward  proof.
We thus assume that $J$ is non-unital. Note then that $J+\C1$ may be regarded as the unitization of $J$.
 	
	(i) Let us first show that $\rho\stackrel{\iota_J^*}{\longmapsto}\rho|_J$ maps $D_A(\phi)$ into $D_J(\phi|_J)$.
	Let  $\psi\in D_A(\phi)$ and suppose that
	$\psi|_J\notin D_J(\phi|_J)$.  Then, by the Hahn-Banach theorem,  there exist $a\in J_{\sa}$ and $t\in\R$ such that
	$\psi(a)>t$ and $\rho(a)\leq t$ for all $\rho\in D_J(\phi|_J)$. It follows from Lemma \ref{DphiphiD} applied to
	$\phi|_J$ and $a$ that  $\phi(b)\leq t$ for all $b\in D_J(a)$. But, by \cite[Remark 2.6]{Arch-PLMS}, $D_J(a)=D_A(a)$ (since
	$a\in J$). 	Hence $\phi(b)\leq t$ for all $b\in D_A(a)$. Lemma \ref{DphiphiD}, applied now to $\phi$ and $a$, implies that $\rho(a)\leq t$ for all $\rho\in D_A(\phi)$.
	Since $\psi\in D_A(\phi)$, we obtain that $\psi(a)\leq t$ which gives a contradiction. Thus $\iota_J^*$ maps $D_A(\phi)$ into $D_J(\phi|_J)$.
	
	Let us prove surjectivity. Since $\iota_J^*$ is weak$^*$-continuous,  the image of $D_A(\phi)$ is a weak$^*$-compact convex subset of
	$D_J(\phi|_J)$.
	For  every  $T\in \M(A,U(J+\C1))$ we have $(\phi\circ T)|_J = \phi|_J\circ T|_J$.
Clearly, every mixing operator in $\M(J)$ has the form $T|_J$ for some $T\in \M(A,U(J+\C1))$. Thus, letting $T$ range through $\M(A,U(J+\C1))$	the functionals $\phi|_J\circ T|_J$ range through a dense subset of $D_J(\phi|_J)$.
	This shows that the image of $D_A(\phi)$ by $\iota_J^*$  is also dense  in $D_J(\phi|_J)$.

	(ii) Clearly $D_A(\phi,U(J+\C1))\subseteq D_A(\phi)$. To prove the opposite inclusion it suffices to show that $u\phi u^*\in D_A(\phi, U(J+\C1))$ for all $u\in U(A)$.
	Let $u\in U(A)$   and set  $\psi=u\phi u^*$. By (i), $\psi|_J\in D_J(\phi|_J)$, so there exists a net of mixing operators
	$(T_i)_i$ in $\M(A,U(J+\C1))$ such that
	\[
	(\phi \circ T_i)|_J=(\phi|_J) \circ (T_i|_J)\stackrel{\mathrm{weak}^*}{\longrightarrow}\psi|_J.
	\]
	Passing to a subnet if necessary, we may assume
	that $\phi\circ T_i\to \psi'\in D_A(\phi,U(J+\C1))$.
	Then $\psi'|_J=\psi|_J$. Moreover, $\|\psi'\|\leq \|\phi\|=\|\phi|_J\|=\|\psi|_J\|$.
	By the uniqueness of the norm-preserving positive extension of a positive functional, we get that $\psi'=\psi$. Thus, $\psi\in  D_A(\phi,U(J+\C1))$.
	
(iii) The image of $D_{A/J}(\phi)$ by $q_J^*$ is the set  $\{\rho\circ q_J: \rho\in D_{A/J}(\phi)\}$. This set is convex, weak* compact, and contains $\phi\circ q_J$. Moreover,
	for $u\in U(A)$ and $\rho\in D_{A/J}(\phi)$ we have $u(\rho\circ q_J)u^*=(v\rho v^*)\circ q_J$, where
$v=q_J(u)\in  U(A/J)$.
	Hence $\{\rho\circ q_J: \rho\in D_{A/J}(\phi)\}$ is invariant under unitary conjugations. It follows that
	\[
	D_A(\phi\circ q_J)\subseteq\{\rho\circ q_J: \rho\in D_{A/J}(\phi)\}.
	\]
	To prove the reverse inclusion it suffices to show that the left side is dense in the right side (since the left side is weak* compact).
	By Proposition \ref{eqn: exp-unitary for A} (ii) (applied in $A/J$), it suffices to show that $e^{ik}\phi e^{-ik}\circ q_J$ belongs to $D_A(\phi\circ q_J)$ for all
	$k\in (A/J)_{\sa}$. But   if $k\in (A/J)_{\sa}$
	then we may find $h\in A_{\sa}$ such that $q_J(h)=k$,  from which it follows that
	\[
	(e^{ik}\phi e^{-ik})\circ q_J = e^{ih}(\phi\circ q_J)e^{-ih}\in D_A(\phi\circ q_J),
	\]
	as desired.
	
	We have  shown that $q_J^*$ maps $D_{A/J}(\phi)$ onto $D_A(\phi\circ q_J)$. Since $q_J^*$ is also injective, the result follows.
\end{proof}

Let $J\subseteq A$ be as above a proper closed two-sided ideal of $A$. Let $(A^*_+)^J$ denote the set of functionals
$\phi\in A^*_+$ such that $\|\phi\|=\|\phi|_J\|$. Let $(A^*_+)_J$ denote the functionals $\phi\in A^*_+$ such that
$\phi(J)=\{0\}$. Recall then that every $\phi\in A^*_+$ can be expressed in the form $\phi=\phi_1+\phi_2$,
with $\phi_1\in (A^*_+)^J$ and $\phi_2\in (A^*_+)_J$ and that this decomposition is unique (see, for example, \cite[2.11.7]{Dix}).

\begin{theorem}\label{MMextensions}
	Let	$A$ be a unital C*-algebra and let $J$ be a proper closed ideal of $A$. Let $\phi\in A_+^*$ and write $\phi=\phi_1+\phi_2$, where $\phi_1,\phi_2\in A_+^*$  are such that $\phi_1\in (A_+^*)^J$
	and $\phi_2\in (A_+^*)_J$.
	\begin{enumerate}
		\item
		$\phi_1$ is maximally mixed if and only if $\phi_1|_J\in J_+^*$ is maximally mixed.
		\item
		$\phi_2$ is maximally mixed if and only if the functional that it induces on $A/J$ is maximally mixed.
		
		\item
		$\phi$ is maximally mixed if and only if both $\phi_1$ and $\phi_2$ are maximally mixed. Moreover, in this case $D_A(\phi)=D_A(\phi_1)+D_A(\phi_2)$.
	\end{enumerate}	
\end{theorem}
\begin{proof}
	(i) Suppose first that $\phi_1$ is maximally mixed. Let $\psi'\in D_J(\phi_1|_J)$.  By Proposition \ref{Dextensions} (i), there exists
	$\psi\in D_A(\phi_1)$ such that $\psi|_J=\psi'$. Since $\phi_1$ is maximally mixed, $\phi_1\in D_A(\psi)$.
	Then, again by Proposition \ref{Dextensions} (i), $\phi_1|_J\in D_J(\psi')$. Thus,   $\phi_1|_J$ is maximally mixed.
	
	Let us prove the converse.  Let $\psi\in D_A(\phi_1)$. Then $\psi|_J\in D_J(\phi_1|_J)$ by Proposition \ref{Dextensions} (i). Since $\phi_1|_J$
	is maximally mixed,  $\phi_1|_J\in D_J(\psi|_J)$. By Proposition \ref{Dextensions} (i), there exists $\phi_1'\in D_A(\psi)$ such that $\phi_1'|_J=\phi_1|_J$.
	Moreover, $\|\phi_1'\|\leq \|\psi\|\leq \|\phi_1\|$. By the uniqueness of the norm-preserving extension of a positive functional, $\phi_1'=\phi_1$. So $\phi_1\in D_A(\psi)$, as desired.
	
	(ii) This is a rather straightforward consequence of Proposition \ref{Dextensions} (iii). Let $\tilde \phi\in (A/J)^*$ be such that $\phi=\tilde\phi\circ q_J$.
	Suppose that $\tilde\phi$ is maximally mixed. By Proposition \ref{Dextensions} (iii), if $\psi\in D_A(\phi)$ then
	$\psi=\tilde \psi\circ q_J$ for some $\tilde\psi\in D_{A/J}(\tilde \phi)$. Since $\tilde \phi$ is maximally mixed,
	$\tilde \phi\in D_{A/J}(\tilde \psi)$. Again by Proposition \ref{Dextensions} (iii), $\phi\in D_A(\psi)$ as desired. Suppose on the other hand that $\phi$
	is maximally mixed. Let  $\tilde \psi\in D_{A/J}(\tilde \phi)$. Then $\tilde\psi\circ q_J\in D_A(\phi)$. Hence,
	$\phi\in D_A(\tilde\psi \circ q_J)$. By Proposition \ref{Dextensions} (iii), $\tilde \phi\in D_{A/J}(\tilde \psi)$ as desired.
	
	(iii) Suppose that $\phi$ is maximally mixed.
	Let $T\in \MM(A^*)$.  Let us show first that $T\phi_1\in (A^*_+)^J$ and
	$T\phi_2\in (A_+^*)_J$.  It is clear that $T\phi_2\in (A_+^*)_J$, since $\phi_2\in (A_+^*)_J$
	and $(A_+^*)_J$ is a Dixmier set. Thus, restricting to $J$ in $T\phi=T\phi_1+T\phi_2$  we obtain that $(T\phi)|_J=(T\phi_1)|_J$. Since $\phi$ is maximally mixed, $\phi\in D_A(T\phi)$, and therefore $\phi|_J\in D_J((T\phi)|_J)$ by Proposition \ref{Dextensions} (i). Hence,
	\[
	\|\phi_1\|=\|\phi|_J\|\leq \|(T\phi)|_J\|=\|(T\phi_1)|_J\|.
	\]
So $\|T\phi_1\|\leq \|\phi_1\|\leq \|(T\phi_1)|_J\|$, which shows that $T\phi_1\in (A_+^*)^J$ (by the definition of $(A^*_+)^J$).

To prove that $\phi_1$ and $\phi_2$ are maximally mixed we proceed as follows:
	Since $\phi$ is maximally mixed,
	there exists $S\in \MM(A^*)$ such that $ST\phi=\phi$. We thus have that $\phi=ST\phi_1+ST\phi_2$.
	Using the last paragraph with $ST$ in place of $T$, we have that $ST\phi_2\in (A_+^*)_J$  and $ST\phi_1\in (A_+^*)^J$. By the uniqueness of the decomposition of $\phi$
	into a functional in $(A_+^*)^J$ and one in $(A_+^*)_J$ we conclude that
	$ST\phi_1=\phi_1$ and $ST\phi_2=\phi_2$. Thus, for any $T\in \MM(A^*)$
	there exists $S\in \MM(A^*)$ such that $ST\phi_1=\phi_1$ and $ST\phi_2=\phi_2$. In view of Lemma \ref{MMoperators}, this shows that
	$\phi_1$ and $\phi_2$ are maximally mixed.

Suppose now that both $\phi_1$ and $\phi_2$ are maximally mixed. Let us show first that
$D_A(\phi)=D_A(\phi_1)+D_A(\phi_2)$. The inclusion $D_A(\phi)\subseteq D_A(\phi_1)+D_A(\phi_2)$
is clear, for if $T\in \MM(A^*)$ then $T\phi=T\phi_1+T\phi_2$, which belongs to $D_A(\phi_1)+D_A(\phi_2)$, and by Lemma \ref{MMoperators} $T\phi$ ranges through all of $D_A(\phi)$. Let $\phi_1'\in D_A(\phi_1)$ and $\phi_2'\in D_A(\phi_2)$ and let us show that $\phi_1'+\phi_2'\in D_A(\phi)$.
Choose $T\in \MM(A^*)$ such that $T\phi_2=\phi_2'$, so that $T\phi=T\phi_1+\phi_2'$.
 Recall that, as shown above,  operators in $\MM(A^*)$ preserve the decomposition of a maximally mixed functional into
	functionals in $(A_+^*)^J$ and $(A_+^*)_J$. Hence,  $T\phi_1\in (A_+^*)^J$. Since $\phi_1'\in D_A(T\phi_1)$, there exists $S\in \MM(A^*)$
	such that $ST\phi_1=\phi_1'$. Moreover, by Proposition \ref{Dextensions} (ii),
	we can choose $S\in \MM(A^*,U(J+\C1))$. Observe then that $S\phi_2'=\phi_2'$ (since
	$\phi_2'$ vanishes on $J$). Hence, $ST\phi=\phi_1'+\phi_2'$, as desired.
	
Continue to assume that $\phi_1$ and $\phi_2$ are maximally mixed and  let us show that $\phi$ is maximally mixed. Let $\phi'\in D_A(\phi)$. Then $\phi'=\phi_1'+\phi_2'$, where $\phi_1'\in D_A(\phi_1)$ and $\phi_2'\in D_A(\phi_2)$. So
\[
D_A(\phi)=D_A(\phi_1)+D_A(\phi_2)=D_A(\phi_1')+D_A(\phi_2')=D_A(\phi'),
\]	
where we use the fact that $\phi_1'$ and $\phi_2'$ are maximally mixed, and the result of the previous paragraph, for the final equality.
Hence, $\phi$ is maximally mixed.
\end{proof}

\begin{corollary}\label{MMunitization}
	Let $A$ be a non-unital C*-algebra and $\phi\in A_+^*$. Then $\phi$ is maximally mixed if and only if its norm preserving positive extension
	to $A^\sim$   is maximally mixed.
\end{corollary}
	
In view of the previous corollary in the sequel we focus our attention  on unital C*-algebras.
Further, since the scalar multiples of a maximally mixed functional are maximally mixed, we
work with states. We denote by $S(A)$ the state space of $A$ and by $S_\infty(A)$ the set of maximally mixed states of $A$.

Let $A$ be a unital C*-algebra. Consider states $\phi\in S(A)$ of the following two types:
\begin{enumerate}
	\item[(A)]
	$\phi$ is tracial,
	\item[(B)]
	$\phi$ factors through a simple quotient $A/M$ without bounded traces.
\end{enumerate}
Not much effort is needed to see that the states of these types are maximally mixed (for tracial states, this is obvious, whereas for type (B) states, it follows from a short argument in Lemma \ref{InfiniteSimpleMM} below); this prompts us to ponder whether all maximally mixed states can be described in terms of these ones.
We show in Theorem \ref{maxmixS} that we are close to getting all maximally mixed states by taking the convex hull of these ones -- although we don't know whether the set of maximally mixed states is convex, see Question \ref{MMconvexquestion} below.

\begin{lemma}\label{InfiniteSimpleMM}
If $B$ is a simple unital C*-algebra with no bounded traces, then for every state $\phi \in S(B)$, $D_B(\phi)=S(B)$, and thus every state of $B$ is maximally mixed.
Therefore every type (B) state of a unital C*-algebra is maximally mixed.
\end{lemma}

\begin{proof}
Suppose for a contradiction that there exists $\psi \in S(B) \setminus D_B(\phi)$.
By the Hahn-Banach theorem,  there exist  $a\in A_{\sa}$ and $t\in \R$ such that $D_A(\phi)(a)\leq t$ (i.e., $s \leq t$ for all $s \in D_A(\phi)(a)$) and $\psi(a)>t$. Translating by a scalar, we may assume that $a$ is positive.
We then know that $\phi(D_A(a))\leq t$ (Lemma \ref{DphiphiD}) and $\psi(a)>t$. But $\|a\|\cdot 1\in D_A(a)$ (by \cite[Th\'eor\`eme 4]{HZ}), and so $\|a\|\leq t$, which
 contradicts that $\psi(a)>t$.

The final statement now follows by Theorem \ref{MMextensions} (ii).
\end{proof}

In the following proposition (and henceforth,  where appropriate) by a type (B) positive functional  we mean a  positive scalar multiple  of a type (B) state.
\begin{proposition}\label{ABstatesConvexity}
Let $A$ be a unital C*-algebra.  Let $\phi\in A^*_+$ be maximally mixed and  let $\psi\in A^*_+$ be either tracial or type (B).
Then $\phi+\psi$ is maximally mixed and $D_A(\phi+\psi)=D_A(\phi)+D_A(\psi)$.
\end{proposition}	

\begin{proof}
	If $\psi$ is tracial then $D_A(\phi+\psi)=D_A(\phi)+\psi$, from which the  result follows at once.
	Suppose then that $\psi$ is type (B), i.e., it factors through a simple quotient $A/M$ without bounded traces.
	Let  $\phi=\phi_1+\phi_2$, where $\phi_1\in (A^*_+)^M$ and
	$\phi_2\in (A_+^*)_M$. Then
	\[
	\phi+\psi=\phi_1+(\phi_2+\psi).
	\]
	By Theorem \ref{MMextensions} (iii), $\phi_1$ is maximally mixed.
	On the other hand, $\phi_2+\psi$ is type (B) (it factors through $A/M$), so by Lemma \ref{InfiniteSimpleMM}, it is maximally mixed.
	Hence, by Theorem \ref{MMextensions} (iii),
  $\phi+\psi=\phi_1+(\phi_2+\psi)$
	is maximally mixed. Moreover, Theorem \ref{MMextensions} (iii) also shows that
	$D_A(\phi+\psi)=D_A(\phi_1)+D_A(\phi_2+\psi)$. But $D_A(\phi_2+\psi)=(\phi_2(1)+\psi(1))S(A)_M$, where $S(A)_M = S(A/M) \circ q_M$ (i.e., all states that factor through $A/M$).
	So
	\begin{align*}
	D_A(\phi+\psi) &=D_A(\phi_1)+\phi_2(1)S(A)_M+\psi(1)S(A)_M\\
	&=D_A(\phi_1)+D_A(\phi_2)+D_A(\psi)\\
    &=D_A(\phi)+D_A(\psi),
	\end{align*}
	using Theorem \ref{MMextensions} (iii) again for the last equality.
\end{proof}

Given a C*-algebra $A$, we denote by  $T(A)$ the set of tracial states on $A$.
\begin{theorem}\label{maxmixS}
	Let $A$ be a unital C*-algebra. Let $S_{(B)}(A)$ denote the set of states on $A$ of type (B).
Then
\[
	\overline{\mathrm{co}(T(A)\cup S_{(B)}(A))}^{\|\cdot\|}\subseteq S_\infty(A)\subseteq \overline{\mathrm{co}(T(A)\cup S_{(B)}(A))}^{\mathrm{weak^*}}.
\]
\end{theorem}

Examples \ref{notweaklyclosed}, \ref{notweaklyclosed2}, and \ref{biggerthangammaclosure} show that both inclusions   in the above theorem can be strict.

\begin{proof}
	It follows by Proposition \ref{ABstatesConvexity} that $\mathrm{co}(T(A)\cup S_{(B)}(A))\subseteq S_\infty(A)$, and so by Theorem \ref{normclosed}
	\[ \overline{\mathrm{co}(T(A)\cup S_{(B)}(A))}^{\|\cdot\|} \subseteq S_\infty(A). \]

	To show that   $S_\infty(A)$ is contained in the weak* closure of $\mathrm{co}(T(A)\cup S_{(B)}(A))$, it suffices to show that for any $\phi\in S(A)$ the Dixmier set $D_A(\phi)$  has nonempty intersection with $\overline{\mathrm{co}(T(A)\cup S_{(B)}(A))}^{\mathrm{weak^*}}$. Suppose, for the sake of contradiction, that this is not the case for some $\phi\in S(A)$.
Then, by the Hahn-Banach theorem,  there exists a
self-adjoint element $a$  and real numbers $t_1<t_2$
	such that $\psi(a)\leq t_1$ for all $\psi\in \mathrm{co}(T(A)\cup S_{(B)}(A))$ and $\psi'(a)\geq t_2$
for all $\psi'\in D_A(\phi)$.  Translating $a$
by a multiple of the unit we can assume that it is positive. Since $D_A(\phi)(a)=\overline{\phi(D_A(a))}$ (Lemma \ref{DphiphiD}), we have that $\phi(a')\geq t_2$ for all $a'\in D_A(a)$. On the other hand,
$\psi(a)\leq t_1$ for every tracial state and every state that factors through a simple quotient without bounded traces. By \cite[Theorem 4.12]{ART}, the distance from $D_A(a)$ to 0 is at most $t_1$. Thus, there exists $a'\in D_A(a)$
such that $\|a'\|<t_2$. This contradicts that
$\phi(a')\geq t_2$.
\end{proof}

Condition (i) of the following corollary has appeared in several papers previously (e.g. \cite{Murphy}, \cite{HZ}, \cite{Ozawa}).
In fact, an improved version of this corollary is \cite[Theorem 5]{HZ}.

\begin{corollary}\label{onlytraces}
Let $A$ be a unital C*-algebra. The following are equivalent:
\begin{enumerate}
\item
Every simple quotient of $A$ has a bounded trace.
\item
All the maximally mixed states of $A$ are tracial.
\item
For every $\phi\in S(A)$ the Dixmier set $D_A(\phi)$ contains a  tracial state.
\end{enumerate}
\end{corollary}	
\begin{proof}
	(i)$\Rightarrow$(ii) follows from the previous theorem, since (i) implies that  $\mathrm{co}(T(A)\cup S_{(B)}(A))=T(A)$, which is already a weak* closed set.
As remarked at the beginning of this section, every Dixmier set $D_A(\phi)$ must contain maximally mixed functionals.  With this in mind,  (ii)$\Rightarrow$(iii) is obvious.
Finally,  assume (iii), and let $A/M$ be a given simple quotient of $A$. Choose any $\phi\in S(A)$  that factors through  $A/M$. Then any trace
in $D_A(\phi)$ factors through $A/M$. We thus have  (i).	
\end{proof}

In the case of simple C*-algebras we obtain a complete description of the maximally mixed positive functionals:
\begin{corollary}\label{MMsimplealgebra}
Let $A$ be a simple C*-algebra.
\begin{enumerate}
\item
If $A$ is unital and has at least one non-zero bounded trace then every maximally mixed positive functional on $A$ is tracial.
\item
If $A$ is unital and has no bounded traces then all the positive functionals on $A$ are maximally mixed.	
\item
If $A$ is non-unital then every maximally mixed positive functional on $A$ is tracial.
\end{enumerate}		
\end{corollary}	

\begin{proof}
(i) follows from Corollary \ref{onlytraces}, while (ii) is Lemma \ref{InfiniteSimpleMM}.
For (iii), note that $A^\sim$ has only one simple quotient, namely $\mathbb C$, and it has a bounded trace. Hence by Corollary \ref{onlytraces}, every maximally mixed state of $A^\sim$ is tracial, and then (iii) follows from Theorem \ref{MMextensions} (i).
\end{proof}

\begin{question}\label{MMconvexquestion}
Let $A$ be a unital C*-algebra.
Is the set $S_\infty(A)$ of  maximally mixed states convex?
\end{question}

A closely related question is the following:
\begin{question}\label{Qsumdixmierset}
Given maximally mixed functionals $\phi,\psi\in A_+^*$, do we have $D_A(\phi+\psi)=D_A(\phi)+D_A(\psi)$?
\end{question}	

	An affirmative answer to  Question \ref{Qsumdixmierset} for all $\phi,\psi\in S_\infty(A)$ also answers  affirmatively
	Question \ref{MMconvexquestion}. Indeed, suppose that Question \ref{Qsumdixmierset}
	has an affirmative answer for all $\phi,\psi\in S_\infty(A)$. Say  we are given  $\phi,\psi\in S_\infty(A)$ and $\phi'\in D_A(\phi)$ and $\psi'\in D_A(\psi)$. Then
\[
D_A(\phi+\psi)=D_A(\phi)+D_A(\psi)=D_A(\phi')+D_A(\psi')=D_A(\phi'+\psi').
\]
Observe that Proposition \ref{ABstatesConvexity} answers Question \ref{Qsumdixmierset} affirmatively in the case that $\psi$ is either tracial or type (B).

Turning to the question of whether the containment $S_\infty(A)\subseteq \overline{\mathrm{co}(T(A)\cup S_{(B)}(A))}^{\mathrm{weak^*}}$ is strict (where $S_{(B)}(A)$ is as defined in Theorem \ref{maxmixS}), it is evident from that theorem that (non-)strictness of this inequality is equivalent to the natural question of whether $S_\infty(A)$ is weak* closed.
The next proposition gives an obstruction to $S_\infty(A)$ being weak* closed -- in fact, it is the only obstruction we have been able to find, see Question \ref{Qweakstarclosed}.

\begin{proposition}\label{MMweakstarclosed}
Let $A$ be a unital C*-algebra such that $S_\infty(A)$ is a weak* closed subset of $S(A)$.
Then the set $X\subseteq \mathrm{Prim}(A)$ of all maximal ideals $M$ such that $A/M$ either is isomorphic to $\C$ or has no bounded traces is a closed subset of $\mathrm{Prim}(A)$.	
\end{proposition}
\begin{proof}
We may assume that $X$ is non-empty. Let $J=\bigcap_{M\in X} M$.
Let $N\in \mathrm{Prim}(A)$ be an adherence point of $X$, i.e, $J\subseteq N$. Then every state on $A$ that factors through $A/N$
is a weak* limit of convex combinations of states that factor through $A/M$, with $M\in X$ (\cite[Proposition 3.4.2 (i)]{Dix}). Notice that $S_\infty(A/M)=S(A/M)$ for all $M\in X$. Thus, all the states of $A$ that factor through $A/M$, with $M\in X$, are maximally mixed.
It follows that  all states factoring through $A/N$ are maximally mixed, and so  all states of $A/N$ are maximally mixed by Theorem \ref{MMextensions} (ii).

Since $N$ is primitive, let $\phi\in S(A/N)$ be a pure state whose GNS representation $\pi_{\phi}$ is faithful.
Then any pure state $\psi$ on $A/N$ is a weak* limit of vector states (with respect to $\pi_{\phi}$) by \cite[Corollary 3.4.3]{Dix}.
By the unitary version of Kadison's Transitivity Theorem (\cite[Theorem 2.8.3 (iii)]{Dix}), each of these vector states is in fact unitarily equivalent to $\phi$, and thus $\psi$ is a weak* limit of unitary conjugates of $\phi$.
By approximating arbitrary states on $A/N$ by convex combinations of pure states, we find that $S(A/N) = D_{A/N}(\phi)$.
This implies that $A/N$ is simple, for otherwise the states factoring through a non-trivial quotient would form a proper Dixmier subset of $D_{A/N}(\phi)$ (recall that
$\phi$ is maximally mixed). From Corollary  \ref{MMsimplealgebra} we see that $A/N$ must either be isomorphic to $\C$ or without bounded traces. Thus, $N\in X$.
\end{proof}		

The examples below show that $S_\infty(A)$ may fail to be weak* closed.
	
\begin{example}\label{notweaklyclosed}
Fix a simple unital C*-algebra $B$ without bounded traces (e.g.,  the Cuntz algebra $\mathcal O_2$).
Let $A$ be the C*-subalgebra of $C([0,1],M_2(B))$ of functions $f$ such that $f(1)\in M_2(\C)\subseteq M_2(B)$.
For each $t\in [0,1]$ let $I_t=\{f\in A : f(t)=0\}$. Then $A/I_t\cong M_2(B)$ for all $0\leq t<1$.
So $I_t$ is a maximal ideal such that $A/I_t$ is simple without bounded traces. The maximal ideal $I_1$ is an adherence point of the set $\{I_t:0\leq t<1\}$. However, $A/I_1\cong M_2(\C)$ has a bounded trace and is not isomorphic to $\C$.  Thus, $S_\infty(A)$ is not weak* closed, by  Proposition \ref{MMweakstarclosed}.	
\end{example}

\begin{example}\label{notweaklyclosed2}	
	Again fix a simple unital C*-algebra $B$ without bounded traces.
	Let $A$ be the C*-subalgebra of $C(\{1,2,\ldots,\infty\},(B\otimes\mathcal K)^\sim)$ of $f$ such that $f(n)\in M_n(B)+\C1$ for all $n\in \N$, where we regard $M_n(B)$ embedded in $B\otimes \mathcal K$ as the top left corner. For each
	$n\in \N$ define
	$I_n=\{f\in A : e_nf(n)=0\}$, where $e_n$ is the unit of $M_n(B)$. Then $I_n$ is a maximal ideal for all $n=1,2,\ldots$ and $A/I_n\cong M_n(B)$ has no bounded traces.  Since $\bigcap_{n} I_n=\{0\}$, the set $\{I_n : n\in \N\}$ is dense in $\mathrm{Prim}(A)$. Consider the ideal $I_\infty=\{f : f(\infty)=0\}$.
Since  $A/I_\infty=(B\otimes \mathcal K)^\sim$ is a primitive C*-algebra, $I_\infty\in \mathrm{Prim}(A)$. But $I_\infty$ is not maximal. By Proposition \ref{MMweakstarclosed}, $S_\infty(A)$ is not weak* closed.

	If one wanted an algebra $A$ with no bounded traces in which $S_\infty(A)$ is not weak* closed, one can simply tensor the example just given with a nuclear, unital, simple, traceless C*-algebra (this operation does not change the ideal lattice, so the same obstruction applies).
\end{example}

\begin{question}\label{Qweakstarclosed}
Is the converse of Proposition \ref{MMweakstarclosed} true?	 That is, let $A$ be unital. Suppose that the set of maximal ideals $M$ such that $A/M$ either is isomorphic to $\C$ or has no bounded traces is a closed subset of $\mathrm{Prim}(A)$. Is $S_\infty(A)$ weak* closed?
\end{question}	

The previous question admits a reduction to the special case of characterizing when $S_\infty(A)$ is all of $S(A)$ (Remark \ref{Questionsequivalence} below).
Before explaining this, we point out the following result:
\begin{proposition}\label{allMM}
Let $A$ be a unital C*-algebra.  Let $X\subseteq \mathrm{Prim}(A)$ be as in Proposition \ref{MMweakstarclosed}.
The following are equivalent:
\begin{enumerate}
\item	
$X= \mathrm{Prim}(A)$;
\item
every pure state of $A$ is either multiplicative or type (B);
\item
$S_{\infty}(A)$ contains the norm-closed convex hull of the pure states.
\end{enumerate}
\end{proposition}
\begin{proof}
(i)$\Rightarrow$(ii): If $\phi$ is a pure state then $\phi$ factors through a primitive quotient $A/N$. By (i),
this quotient is simple and either isomorphic to $\C$	or without traces. In the first case $\phi$ is muliplicative and in the second case it is of type (B).

(ii)$\Rightarrow$(iii): This follows from Theorem \ref{maxmixS}.

(iii)$\Rightarrow$(i): Let $N\in \mathrm{Prim}(A)$ and choose a pure state $\phi\in S(A/N)$ whose GNS representation is faithful.
Then, since $\phi$ is maximally mixed, it follows as in the proof of Proposition \ref{MMweakstarclosed} that $A/N$ is simple.  Since every pure state of $A/N$ is maximally mixed, it follows from Corollary \ref{MMsimplealgebra} that $A/N$ must either be $\C$ or have no bounded traces.
Thus, $N\in X$.
\end{proof}	

\begin{question}\label{QallMM}
Let $A$ be a unital C*-algebra that satisfies the equivalent conditions of Proposition \ref{allMM}. Does it follow that $S_\infty(A)=S(A)$?
\end{question}

\begin{remark}\label{Questionsequivalence}
Questions \ref{Qweakstarclosed} and \ref{QallMM} are equivalent. For if Question \ref{Qweakstarclosed}
has been answered affirmatively and $A$ satisfies the equivalent conditions of Proposition \ref{allMM}, then $S_\infty(A)$
is weak* closed, and so by Proposition \ref{allMM} (iii)  $S_\infty(A)=S(A)$.  Suppose conversely that Question \ref{QallMM} has been answered
affirmatively.  Let $A$ be a unital C*-algebra. Let $X\subseteq \mathrm{Prim}(A)$ be as in Proposition \ref{MMweakstarclosed} and
suppose that $X$ is a closed set. If $X=\varnothing$ then, by Corollary \ref{onlytraces}, $S_{\infty}(A)=T(A)$ (the set of tracial states on $A$),  which is weak$^*$ closed.
Assume $X\neq \varnothing$. Let $J=\bigcap_{M\in X} M$. Then the C*-algebra $A/J$ satisfies the equivalent conditions in Proposition \ref{allMM}
(check (i)). Hence,  all of its states are maximally mixed. By Theorem \ref{MMextensions} (ii), $S(A)_J\subseteq S_\infty(A)$. Observe also that every state of $A$ of type (B) is in $S(A)_J$. Then,
\[
\mathrm{co}(T(A),S(A)_J)\subseteq S_\infty(A)\subseteq \overline{\mathrm{co}(T(A),S(A)_J)}^{\mathrm{weak}^*},
\]
where the first inclusion follows from Proposition \ref{ABstatesConvexity} and the second inclusion
from  Theorem \ref{maxmixS}.  But $\mathrm{co}(T(A),S(A)_J)$ is weak* closed. So  $S_\infty(A)=\mathrm{co}(T(A),S(A)_J)$, is weak* closed.
\end{remark}

In the next section we answer affirmatively Questions \ref{MMconvexquestion}, \ref{Qsumdixmierset}, and \ref{Qweakstarclosed} for C*-algebras with the Dixmier property.

\section{C*-algebras with the  Dixmier property}\label{DPcase}
In this section, we find further properties of the set of maximally mixed states in the case of C*-algebras with the Dixmier property (defined below).

Let $A$ be a unital C*-algebra, let $a\in A$ and let $\phi\in S(A)$.
Acting by conjugation, the unitary group $U(A) $ induces a group of isometric affine transformations of the convex set
$D_A(a)$, and similarly for $D_A(\phi)$. An element $z\in D_A(a)$ is a fixed point for the group of conjugations of $D_A(a)$ if and only if it belongs to the centre $Z(A)$. An element $\tau \in D_A(\phi)$ is a fixed point for the group of conjugations on $D_A(\phi)$ if and only if it is a tracial state, i.e., $\tau\in T(A)$.

The C*-algebra A is said to have the \emph{Dixmier property} if
$D_A(a) \cap Z(A)$ is non-empty for all $a\in A$, and $A$ is said to have the \emph{singleton Dixmier property} if $D_A(a)\cap Z(A)$ is a singleton set for all $a \in A$ (see \cite{ART} and the many references cited therein). On the other hand, we have seen in Corollary \ref{onlytraces} that $D_A(\phi) \cap T(A)$ is non-empty for all
$\phi\in S(A)$ if and only if every simple quotient of $A$ has a tracial state. Since a unital simple C*-algebra has the Dixmier property if and only if it has at most
one tracial state \cite{HZ}, we see that the Dixmier property is neither necessary nor sufficient for the equivalent properties of Corollary \ref{onlytraces} to hold.  Indeed, it is shown in \cite[Proposition 1.4]{ART} that $A$ has the Dixmier property and also satisfies the conditions of Corollary \ref{onlytraces} if and only if it has the singleton Dixmier property.

If a Dixmier set $D_A(a)$ does not contain a central element then there is no natural ``second prize'' at which to aim. In contrast, if a Dixmier set $D_A(\phi)$ does not contain  a tracial state then we may nevertheless study the maximally mixed states, which we have already seen to be guaranteed to exist in $D_A(\phi)$. In this section we study the maximally mixed states in the case where $A$ has the Dixmier property but not necessarily the singleton Dixmier property.

Henceforth in this section we assume that $A$ is a unital C*-algebra with the Dixmier property.

Let $\hat Z$ denote the spectrum of $Z(A)$. Since C*-algebras with the Dixmier property are weakly central (e.g., see \cite{ART}), we can identify $\hat Z$ with the set of maximal ideals of $A$. We denote the latter set by $\mathrm{Max}(A)$.

To analyze the maximally mixed states for such $A$, we will make frequent use of a description of $D_A(a) \cap Z(A)$ (for $a$ self-adjoint) found in \cite{ART} (see \cite[Corollary 4.5]{ART} and the discussion between Theorem 2.6 and Corollary 2.7 of \cite{ART}, with details in the proof of Theorem 2.6).
For $a \in A$ self-adjoint, define $f_a,g_a\colon \hat Z \to \R$ by
\[
f_a(M) := \begin{cases} \min \mathrm{sp}(q_M(a)), \quad &\text{if }A/M\text{ has no bounded traces}; \\ \tau_M(a),\quad &\text{otherwise}, \end{cases}
\]
where $\tau_M$ is the (necessarily unique) tracial state on $A$ which factors through $A/M$.
Likewise,
\begin{equation}
\label{g_aDef}
 g_a(M) := \begin{cases} \max \mathrm{sp}(q_M(a)), \quad &\text{if }A/M\text{ has no bounded traces}; \\ \tau_M(a),\quad &\text{otherwise}. \end{cases}
\end{equation}
Then $f_a$ is upper semicontinuous, $g_a$ is lower semicontinuous, $f_a \leq g_a$, and, identifying $Z(A)=C(\hat Z)$ now,
\[ D_A(a) \cap Z(A) = \{z \in C(\hat Z): z=z^* \text{ and } f_a \leq z \leq g_a\}. \]

Let us say that two maximally mixed bounded functionals $\phi$ and $\psi$ are equivalent if they generate the same Dixmier set,
i.e., $D_A(\phi)=D_A(\psi)$.

\begin{proposition}\label{measuresbijection}
Let $A$ be a unital C*-algebra with the Dixmier property. The equivalence classes of maximally  mixed, bounded functionals on $A$
are in bijective correspondence with the bounded functionals on the center of $A$. The correspondence
is given by the restriction map $\phi\mapsto \phi|_{Z(A)}$, for $\phi$ maximally mixed.
\end{proposition}
\begin{proof}
Any two equivalent functionals agree on the center, so the mapping is well defined on equivalence classes.
To see that it is onto, fix a functional $\mu\in Z(A)^*$. The set of all $\phi\in A^*$  whose restriction
to $Z(A)$  is $\mu$ is a weak* compact Dixmier set. It thus must  contain   maximally mixed functionals.

Let us now show that the mapping is injective. Let $\phi,\psi \in A^*$  be two maximally mixed  functionals that agree on $Z(A)$. Suppose for a contradiction that
$D_A(\phi)\neq D_A(\psi)$. Then
$D_A(\phi)$ and $D_A(\psi)$  are  disjoint. By the Hahn-Banach theorem, we can find $a\in A$ and real numbers $t_1<t_2$
such that $\mathrm{Re}(\phi'(a))\leq t_1$ for all $\phi'\in D_A(\phi)$ and $\mathrm {Re}(\psi'(a))\geq t_2$ for all $\psi'\in D_A(\psi)$. By Lemma \ref{DphiphiD},
$\mathrm{Re}(\phi(a'))\leq t_1$ and $\mathrm{Re}(\psi(a'))\geq t_2$ for all $a'\in D_A(a)$. This holds in particular for $a'\in D_A(a)\cap Z(A)$.
This contradicts that $\phi$ and $\psi$ agree on $Z(A)$.
\end{proof}

\begin{remark} The previous proposition implies that if $A$ has the Dixmier property then  $D_A(\phi)$, for $\phi\in S(A)$, contains a unique equivalence class
of maximally mixed states; namely, the maximally mixed  states that agree with $\phi$  on $Z(A)$. This is in general not true
for C*-algebras without the Dixmier property. Take for example  $A$ to be a  simple unital C*-algebra with at least two tracial states and let $\phi$ be a pure state of $A$. Then
$D_A(\phi)$ is the set of all states, so it contains distinct tracial states (which are inequivalent maximally mixed states).
\end{remark}

We need the following little lemma in the proceeding theorem.

\begin{lemma}\label{Regularity}
Let $X$ be a Hausdorff topological space, let $\mu$ be a Radon probability measure on $X$, and let $f\colon X \to \R$ be a bounded lower semicontinuous function.
Then
\[ \int_X f\,d\mu = \sup \int_X g\,d\mu, \]
where the supremum is taken over upper semicontinuous functions $g\colon X \to \R$ which are (pointwise) dominated by $f$.
\end{lemma}

\begin{proof}
Without loss of generality, $f \geq 0$.
We may approximate $f$ uniformly by simple lower semicontinuous functions, i.e., positive scalar linear combinations of characteristic functions of open sets.
Thus, it suffices to handle the case that $f$ is the characteristic function of an open set, say $f=\chi_U$.

In this case, since $\mu$ is inner regular, $\mu(X)$ is the supremum of measures of compact sets $K$ contained in $U$, so
\begin{align*}
\int_X f\,d\mu &= \mu(X) \\
&= \sup_K \mu(K) \\
&= \sup_K \int_X \chi_K\,d\mu,
\end{align*}
where the suprema are taken over compact sets contained in $U$; but now we are done, since each $\chi_K$ is upper semicontinuous.
\end{proof}

\begin{theorem}\label{DPcriterion}
Let $A$ be a unital C*-algebra with the Dixmier property. Let $\phi\in A_+^*$.  The following are equivalent:
\begin{enumerate}
\item
$\phi$ satisfies that
\begin{equation}\label{criterion}
\phi(a)\leq \sup \{\phi(z): z\in D_A(a)\cap Z(A)\}\qquad (a\in A_+).
\end{equation}
\item
$\phi$ is maximally mixed.
\end{enumerate}
\end{theorem}
\begin{proof}
(i)$\Rightarrow$(ii).
Suppose for a contradiction that there exists
$\psi\in D_A(\phi)$ such that $\phi\notin D_A(\psi)$.
Then there exist a  self-adjoint element $a$ and $t\in \R$  separating $D_A(\psi)$ and $\phi$. That is, $\psi'(a)\leq t$ for all $\psi'\in D_A(\psi)$ and $\phi(a)>t$.
Translating $a$ by a scalar multiple of the unit we may assume that it is positive. By Lemma \ref{DphiphiD}, we get that $\psi(a')\leq t$  for all $a'\in D_A(a)$.
From  $\psi\in D_A(\phi)$ we deduce that  $\psi(a')=\phi(a')$ for all $a'\in Z(A)$. Hence
\begin{align*}
\phi(a)&\leq \sup\{ \phi(a') : a'\in D_A(a)\cap Z(A)\}\\
&=\sup\{ \psi(a') : a'\in D_A(a)\cap Z(A)\}\leq t.
\end{align*}
This contradicts that $\phi(a)>t$.

(ii)$\Rightarrow$(i). We may assume that $\phi\neq 0$ and then, multiplying it by a scalar, that it is a state. First, let us  show  that if a maximally mixed state $\phi$ satisfies \eqref{criterion} then so do all the states equivalent to it.
Let $\phi$ be a state that satisfies \eqref{criterion}  and let  $\psi\in D_A(\phi)$. Say $\psi=\lim_i\phi\circ T_i$, where $(T_i)_i$ is a net of mixing operators in $\M(A)$.
Let $a \in A_+$.
Since $D_A(T_i a)\subseteq D_A(a)$,
\begin{align*}
\phi(T_i a) &\leq \sup \{\phi(z): z\in D_A(T_i a)\cap Z(A)\}\\
&\leq \sup \{\phi(z): z\in D_A(a)\cap Z(A)\}.
\end{align*}
Hence
\begin{align*}
\psi(a) &=\lim_i \phi(T_i\cdot a)\\
&\leq \sup \{\phi(z): z\in D_A(a)\cap Z(A)\}\\
&=\sup \{\psi(z): z\in D_A(a)\cap Z(A)\},
\end{align*}
where the last equality is valid since $\phi$ and $\psi$ agree on $Z(A)$.

By Proposition \ref{measuresbijection},  it now suffices to show that every probability (Radon) measure  $\mu$
on the center can be extended to a state $\phi$ on $A$  satisfying \eqref{criterion}. We do this next.

For each $a\in A_{\sa}$ let us define $p_\mu(a)\in [0,\infty)$ by
\[
p_\mu(a):=\int_{\widehat Z}g_{|a|}(M)\,d\mu(M),
\]
where $g_{|a|}\colon \widehat{Z}\to [0,\infty)$ is the lower semicontinuous function on the spectrum of the center associated to  $|a|$ (as in \eqref{g_aDef} with $|a|$ in place of $a$). Let us show that $p_\mu$ is a seminorm. Clearly $p_\mu (t a)=|t|p_\mu(a)$ for any $t\in \R$. To prove the triangle inequality 
it suffices to show that $g_{|a+b|}\leq g_{|a|}+g_{|b|}$ for all $a,b\in A_{\sa}$. Let us evaluate both sides of this inequality on an ideal $M\in \mathrm{Max}(A)$ such that $A/M$ has no bounded traces. Set $\bar a=q_M(a)$ and $\bar b=q_M(b)$ (the images of $a$ and $b$ in $A/M$).
Then we must show that $\||\bar a+\bar b|\|\leq \||\bar a|\|+\||\bar b|\|$. But this is clear from the triangle inequality for $\|\cdot\|$ and the fact that the norm of an element is equal to the norm of its absolute value. Suppose now that $M$ is such that $A/M$ has bounded traces. Let $\tau_M$ be the unique tracial state on $A$  factoring through $A/M$. Then we must show that 
\begin{align}\label{tautriangle}
\tau_M(|a+b|)\leq \tau_M(|a|)+\tau_M(|b|).
\end{align} 
Let $p\in A^{**}$ be a projection such that 
$p(a+b)p=(a+b)_+$.
Multiplying by $p$ on the left and on the right of $a+b\leq   a_+ + b_+$ we get
$
(a+b)_+\leq p  a_+p+p  b_+p
$.
Evaluating $\tau_M$ (extended to a normal trace on $A^{**}$) on both sides we get 
\[
\tau_M((a+b)_+)\leq \tau_M (a_+)+\tau_M(b_+).
\]
The same inequality, applied to $-a$ and $-b$, yields that
\[
\tau_M((a+b)_-)\leq \tau_M (a_-)+\tau_M(b_-).
\]
Now adding both inequalities we get \eqref{tautriangle}, as desired. 
Thus, $p_\mu$ is a seminorm. Since $g_{|a|} \leq \|a\|$, we  also have that $p_\mu(a)\leq \|a\|$ for all $a\in A_{\sa}$.

For any self-adjoint central element $z$ we have that
\[
\Big|\int z(M)\,d\mu(M)\Big|\leq \int|z(M)|\,d\mu(M)=p_\mu(z).
\]
So we can extend $\mu$ by the Hahn-Banach extension theorem to a self-adjoint functional $\phi$ on $A$ such that
\[
|\phi(a)|\leq p_\mu(a)\qquad (a\in A_{\sa}).
\]
Notice that $\phi(1)=1$ and that $\|\phi\|\leq 1$, since $p_\mu(a)\leq \|a\|$ for all $a\in A_{\sa}$.
Hence, $\phi$ is a state.

 Let $a\in A_+$. To establish \eqref{criterion}, we will show that $p_\mu(a)$ is dominated by the right-hand side of \eqref{criterion} (though we don't need it, in fact this implies that these two quantities are equal, as the reverse inequality is straightforward).
Let $\epsilon>0$. Since $g_a$ is lower semicontinuous, it follows from Lemma \ref{Regularity} that we may find an upper semicontinuous function $w \in C(\hat Z)$ such that $w \leq g_a$ and $\int w(M)\,d\mu(M) > \int g_a(M)\,d\mu(M)-\epsilon$.
By the Katetev-Tong insertion theorem, we may find a continuous function $z_0 \in C(\hat Z)_+$ such that
\[ \begin{array}{c} f_a \\ w \end{array} \leq z_0 \leq g_a, \]
and therefore
\[ \int z_0(M)\,d\mu(M) \geq \int w(M)\,d\mu(M) > \int g_a(M)\,d\mu(M)-\epsilon = p_\mu(a)-\epsilon. \]
Thus
\[ p_\mu(a) \leq \sup \{\phi(z): z\in D_A(a)\cap Z(A)\}, \]
as required.
\end{proof}

\begin{corollary}
\label{DPconvex}
Let $A$ be a  unital C*-algebra with the Dixmier property.
\begin{enumerate}
\item
$S_\infty(A)$ is convex and weakly closed.	
\item
We have  $D_A(\phi+\psi)=D_A(\phi)+D_A(\psi)$ for all   $\phi,\psi\in A_+^*$ maximally mixed.
\end{enumerate}
\end{corollary}

\begin{proof}
(i) Let us show first that $S_\infty(A)$ is convex. Since a scalar multiple of a maximally mixed functional is again maximally mixed, it suffices to show that if $\phi,\psi\in A_+^*$
are maximally mixed, then $\phi+\psi$ is maximally mixed. So let $\phi,\psi\in A_+^*$
be maximally mixed. We show that $\phi+\psi$ satisfies \eqref{criterion}.  Let $a\in A_+$ and $\epsilon>0$.
Since $\phi$ and $\psi$   satisfy \eqref{criterion},  there exist $x,y\in D_A(a)\cap Z(A)$ such that
\[
\phi(a)\leq \phi(x)+\epsilon\hbox{ and } \psi(a)\leq \psi(y)+\epsilon.
\]
By the structure of $D_A(a)\cap Z(A)$ for self-adjoint $a$ we know that it is a lattice. So we can choose $z\in D_A(a)\cap Z(A)$ such that $x,y\leq z$.
Then $(\phi+\psi)(a)\leq (\phi+\psi)(z)+2\epsilon$. This shows that $\phi+\psi$ satisfies  \eqref{criterion} and is therefore maximally mixed.

Since $S_\infty(A)$ is convex and norm closed (Theorem \ref{normclosed}), it is also weakly closed (i.e., closed in the $\sigma(A^*,A^{**})$ topology). 

(ii) The inclusion $D_A(\phi+\psi)\subseteq D_A(\phi)+D_A(\psi)$ is straightforward: if  $T\in \MM(A^*)$ then
\[
T(\phi+\psi)=T\psi+T\psi\in D_A(\phi)+D_A(\psi),
\]
and letting $T$ range through $\MM(A^*)$,  $T(\phi+\psi)$ ranges through all of
$D_A(\phi+\psi)$ (Lemma \ref{MMoperators}).

Let $\phi,\psi\in A_+^*$  be maximally mixed  and suppose, for a contradiction,  that there exist  $\phi'\in D_A(\phi)$ and $\psi'\in D_A(\psi)$ such that $\phi'+\psi'\notin D_A(\phi+\psi)$.
Then  there exist $a\in A_{\sa}$ and $t\in \R$ such that $\rho(a)\leq t$ for all $\rho\in D_A(\phi+\psi)$ while $(\phi'+\psi')(a)>t$. Translating $a$ by a scalar multiple of the unit, and changing $t$ accordingly, we may assume that $a$ is positive. By Lemma \ref{DphiphiD},
$(\phi+\psi)(b)\leq t$ for all $b\in D_A(a)$. Since $\phi+\psi$ and $\phi'+\psi'$ agree on $Z(A)$,
we obtain that $(\phi'+\psi')(b)\leq t$ for all $b\in D_A(a)\cap Z(A)$. But  $\phi'+\psi'$ is maximally mixed by the proof of (i). It follows by Theorem \ref{DPcriterion} that
 $(\phi'+\psi')(a)\leq t$, which contradicts our choice of $a$ and $t$.
\end{proof}
\begin{remark}
The C*-algebras in Examples \ref{notweaklyclosed} and \ref{notweaklyclosed2} both have the Dixmier property (this can be deduced from \cite[Theorem 1.1]{ART}). So $S_\infty(A)$ may fail to be weak* closed for C*-algebras with the Dixmier property.
\end{remark}

\begin{theorem}\label{DPweaklyclosed}
Let $A$ be a unital C*-algebra with the Dixmier property.
The following are equivalent.
\begin{enumerate}
\item The set $S_\infty(A)$ is weak* closed;
\item The set of maximal ideals $M$ such that $A/M$ either is isomorphic to $\C$ or  has no bounded traces  is a closed subset of $\mathrm{Prim}(A)$;
\item For each self-adjoint $a \in A$, the set $D_A(a) \cap Z(A)$ contains a maximal element.
\end{enumerate}
\end{theorem}
	
\begin{proof}
(i)$\Rightarrow$(ii): This is Proposition \ref{MMweakstarclosed} (no Dixmier property required).

(ii)$\Rightarrow$(iii):
By translating, we may assume that $a \geq 0$.
Let $X$ denote the set of maximal ideals $M\in \mathrm{Max}(A)$ such that $A/M$ either is isomorphic to $\C$ or  has no bounded traces, and we assume that this set is closed in $\mathrm{Prim}(A)$.
It is evident from the description of $D_A(a) \cap Z(A)$, at the beginning of this section, that we need only show that the function $g_a:\hat Z \to \R$ from \eqref{g_aDef} is continuous.
Since $g_a$ is always lower semicontinuous, it remains	to show that it is upper semicontinuous.
Let $t>0$.
Set
\[ Y:=\{M\in \mathrm{Max}(A) : T(A/M)\neq \varnothing\}, \]
which is closed in $\mathrm{Max}(A)$ by \cite[Theorem 2.6]{ART}; for $M \in Y$, $A$ has a unique tracial state $\tau_M$ that factors through $A/M$.
Since $\tau_M$ depends weak* continuously on $M\in Y$ (\cite[Theorem 2.6]{ART}),
\[
\{M\in Y:\tau_M(a)\geq t\}
\]
is closed in $\mathrm{Max}(A)$. Also, $\{M\in \mathrm{Prim}(A):\|q_M(a)\|\geq t\}$
is a compact subset of $\mathrm{Prim}(A)$ (\cite[Proposition 
3.3.7]{Dix}), from which (along with that $X$ is closed in $\mathrm{Prim}(A)$) we deduce that
\[
\{M\in \mathrm{Prim}(A):\|q_M(a)\|\geq t\}\cap X
\]
is compact. Since $\mathrm{Max}(A)$ is Hausdorff, the set above is also closed in $\mathrm{Max}(A)$. Therefore,
\[
\{M\in Y:\tau_M(a)\geq t\}\cup (\{M\in \mathrm{Prim}(A):\|q_M(a)\|\geq t\}\cap X)
\]
is closed in $\mathrm{Max}(A)$. But this set is $g_a^{-1}([t,\infty))$, and therefore, $g_a$ is upper semicontinuous.

(iii)$\Rightarrow$(i):
For each self-adjoint element $a \in A$, let $z_a$ denote the maximal element of $D_A(a) \cap Z(A)$, which exists since we are assuming (iii).
Given a state  $\phi$, the inequality \eqref{criterion}
 is equivalent to  $\phi(a)\leq \phi(z_a)$ for all $a\in A_+$.
 The latter inequality is clearly preserved under weak* limits. By Theorem \ref{DPcriterion}, $S_\infty(A)$ is weak* closed.
\end{proof}
	
We recover as a corollary Alberti's theorem on the maximally mixed states
of a von Neumann algebra (\cite[Theorem 5.2]{Alberti} and \cite[Theorem 4-12]{Alb-Uhl}):
\begin{corollary}
Let $A$ be a von Neumann algebra. Then $S_\infty(A)$ agrees with the weak* closure of the convex hull of the set of tracial states and type (B) states.
\end{corollary}
\begin{proof}
Let $A=A_{\mathrm f}\bigoplus A_{\mathrm{pi}}$	be the decomposition of $A$ into a finite and a properly infinite von Neumann algebra.
Let $J\subseteq A_{\mathrm{pi}}$ denote the strong Jacobson radical of $A_{\mathrm{pi}}$, i.e, the intersection of all maximal ideals of $A_{\mathrm{pi}}$.
 By \cite[Proposition 2.3]{Halpern}, $N\in \mathrm{Prim}(A_{\mathrm{pi}})$ is a maximal ideal of $A_{\mathrm{pi}}$ if and only if $J\subseteq N$. It follows that
 $N\in \mathrm{Prim}(A)$ is maximal (in $A$) and such that $A/N$ has no bounded traces if and only if $A_{\mathrm{f}}\oplus J\subseteq N$.
 This set is thus a closed subset of $\mathrm{Prim}(A)$. On the other hand,  the set of $M\in \mathrm{Prim}(A)$ such that
 $A/M\cong \C$ is a closed subset of $\mathrm{Prim}(A)$ (indeed, these are the $M\in \mathrm{Prim}(A)$ that contain the ideal generated by the commutators of $A$, which is the smallest ideal the quotient by which is abelian). So the union of these two sets is closed.  Moreover, $A$
 has the Dixmier property (by Dixmier's approximation theorem). Thus, by the previous theorem, $S_\infty(A)$ is weak* closed.
 The result then follows from Theorem \ref{maxmixS}.
\end{proof}

We end this section by taking advantage of the insight we have gained in the case of the Dixmier property, to provide some examples alluded to earlier.
The first example shows that the set of maximally mixed states may be larger than the norm-closed convex hull of the tracial states and type (B) states.

\begin{example}\label{biggerthangammaclosure}
Let $B$ be a simple unital C*-algebra with no bounded traces, and set $A:=C([0,1],B)$.
If $\phi$ is in the norm-closed convex hull of the type (B) states, then the state that $\phi$ induces on the centre is in the norm-closed convex hull of point-masses, and therefore corresponds to a discrete measure on $[0,1]$.
However, $A$ has the Dixmier property by \cite[Theorem 2.6]{ART} and hence  $S_\infty(A)$ is weak* closed by Theorem \ref{DPweaklyclosed} ((ii) $\Rightarrow$ (i)). Since every pure state of $A$ is of type (B), it follows from Theorem~\ref{maxmixS} that $S_{\infty}(A)=S(A)$.
So the norm-closed convex hull of the type (B) states (and tracial states, as there are none) is only a small part of $S_\infty(A)$ in this case.
\end{example}

The next example addresses the converse to Theorem \ref{jordan}.

\begin{example}
\label{JordanConverseFails}
Let $A$ be a simple unital C*-algebra with no bounded traces. Then $A$ has the Dixmier property (\cite{HZ}).
Let $\phi$ be a nonzero self-adjoint functional on $A$ such that $\phi(1)=0$.
Then $\phi$ is not maximally mixed, because if it were, then since the zero functional is maximally mixed, it would follow by Proposition \ref{measuresbijection} that $D_A(\phi)=D_A(0)=\{0\}$.
However, by Corollary \ref{MMsimplealgebra} (ii), both the positive and negative parts of $\phi$ are maximally mixed.
\end{example}

\section{Hausdorff primitive spectrum}\label{Hausdorffcase}

Here we impose a different property -- Hausdorffness of the primitive ideal space -- to make the study of the structure of $S_\infty(A)$ tractable.

Given a C*-algebra $A$, we continue to denote by $T(A)$ the set of tracial states on $A$.
\begin{theorem}
Let $A$ be a unital C*-algebra with Hausdorff primitive spectrum.
\begin{enumerate}
\item
Suppose that $T(A)=\varnothing$. Then every state of $A$ is maximally mixed.

\item
Suppose that $T(A)\neq \varnothing$. Then the set
\[
Y:=\{M\in \mathrm{Max}(A) : T(A/M)\neq \varnothing\}
\]	
is non-empty and closed in $\mathrm{Max}(A)$ and
\[
S_\infty(A)=\mathrm{co}(T(A)\cup S(A)^J),
\]
where $J:= \bigcap_{M\in Y} M$ is a proper closed ideal of $A$, and $S(A)^J$ consists of all states in $S(A)$ which arise as extensions of states in $S(J)$.
\item
Questions \ref{MMconvexquestion} and \ref{Qsumdixmierset} have affirmative answers for $A$.
\end{enumerate}
\end{theorem}	
\begin{proof}
Observe first that, since $\mathrm{Prim}(A)$ is Hausdorff, $\mathrm{Prim}(A)=\mathrm{Max}(A)=\mathrm{Glimm}(A)$,
and these spaces are all homeomorphic to $\mathrm{Max}(Z(A))$ via the assignment $M\mapsto M\cap Z(A)$. For each maximal ideal $N$ of $Z(A)$, let $\phi_N$ be the unique pure state of $Z(A)$ with kernel equal to $N$.

(i) Since the continuous functions on the compact Hausdorff space $\mathrm{Prim}(A)$ separate the points, it follows from the Dauns-Hofmann theorem that $A$ is a central C*-algebra. Combining this with the fact that $T(A)$ is empty, we obtain from \cite[Theorem 2.6]{ART} that $A$ has the Dixmier property. Every pure state of $A$ is of type (B), so by Theorem \ref{maxmixS}, $S_\infty(A)$ is weak* dense in $S(A)$. It follows from Theorem \ref{DPweaklyclosed} that every state of $A$ is maximally mixed.

 (ii) Since $T(A)$ is non-empty, it contains an extreme point $\tau$ by the Krein-Milman theorem. By \cite[Lemma 2.4]{ART},
 $\tau|_{Z(A)}$ is a pure state of $Z(A)$ and hence annihilates $M\cap Z(A)$ for some $M\in \mathrm{Max}(A)$. By the Cauchy--Schwartz inequality for states, $\tau$ annihilates the Glimm ideal $(M\cap Z(A))A$. But, as noted above, $(M\cap Z(A))A=M$ and so $\tau$ induces a tracial state of $A/M$. Thus $Y$ is non-empty. Moreover, $\tau(J)=\{0\}$ and so, by the Krein-Milman theorem, every tracial state of $A$ annihilates $J$.

  To show that $Y$ is closed, suppose that $(M_i)$ is a net in $Y$ that is convergent to $M\in\mathrm{Max}(A)$.
  For each $i$, let $\tau_i$
be a tracial state of $A$ that vanishes on $M_i$.
  Since $T(A)$ is weak$^*$-compact, there exist $\tau\in T(A)$ and a subnet $(\tau_{i_j})$ such that $\tau_{{i_j}}\to_j \tau$. Then
$$\tau|_{Z(A)} = \lim_j  \phi_{M_{i_j}\cap Z(A)} = \phi_{M\cap Z(A)}.$$
It follows from the Cauchy--Schwartz inequality for states that $\tau$ annihilates the Glimm ideal $(M\cap Z(A))A$ and so $\tau$ induces a tracial state of $A/M$ as before. Thus $M\in Y$, as required.

Since $Y$ is closed, every maximal ideal of $A/J$ has the form $M/J$ for some $M\in Y$ and hence every simple quotient of $A/J$ has a tracial state. It follows by Corollary \ref{onlytraces} that $S_\infty(A/J)=T(A/J)$.
Letting $S_2$ be the set of maximally mixed states of $S_\infty(A)$ which factor through $A/J$, it follows by Theorem \ref{MMextensions} (ii) that
\[ S_2 = S_\infty(A/J) \circ q_J = T(A/J) \circ q_J = T(A). \]

Under the Dauns--Hofmann isomorphism between $Z(A)$ and $C(\mathrm{Prim}(A))$, $Z(J)$ corresponds to $C_0(\mathrm{Prim}(J))$, where $\mathrm{Prim}(J)$ is identified with an open subset of $\mathrm{Prim}(A)$ (namely $\mathrm{Prim}(A)\setminus Y$ ) in the usual way.
It follows that $Z(J)$ separates the primitive ideals of $J+\C1$ and hence $J+\C1$ is a central C*-algebra.
Since $J$ has no tracial states (this follows from \cite[Lemma 2.2]{ART}), $J+\C1$ has a unique tracial state, namely the one factoring through the quotient $(J+\mathbb{C}1)/J$.
Hence by \cite[Theorem 2.6]{ART}, $J+\mathbb{C}1$ has the Dixmier property.
We also have that $\mathrm{Prim}(J+\mathbb{C}1)=\mathrm{Max}(J+\mathbb{C}1)$, with every simple quotient being either traceless or isomorphic to $\C$, and thus by Theorem \ref{DPweaklyclosed}, $S_\infty(J+\mathbb{C}1)$ is weak* closed.
Since every pure state is either type (B) or tracial, it now follows from Theorem \ref{maxmixS} that $S_\infty(J+\mathbb{C}1)=S(J+\mathbb{C}1)$.
By Theorem \ref{MMextensions} (i) (used once with $J \vartriangleleft J+\mathbb{C}1$ and again with $J \vartriangleleft A$), $S(A)^J \subseteq S_\infty(A)$.
Thus, letting $S_1$ be the set of maximally mixed states of $A$ which are extensions of states from $J$, we have
\[ S_1 = S(A)^J. \]

By Theorem \ref{MMextensions} (iii), we have
\[
S_\infty(A) = \mathrm{co}(S_1 \cup S_2) = \mathrm{co}(S(A)^J \cup T(A)), \]
as required.

(iii)
It is evident in both the cases covered by (i) and (ii) that $S_\infty(A)$ is convex.
Now let $\phi,\psi \in A_+^*$ be maximally mixed, and let's argue that $D_A(\phi+\psi)=D_A(\phi)+D_A(\psi)$.
In case (i), we saw that $A$ has the Dixmier property, so  this holds by Corollary \ref{DPconvex} (ii).

In case (ii), write $\phi = \phi_1+\phi_2$ and $\psi=\psi_1+\psi_2$ where $\phi_1,\psi_1$ are positive tracial functionals and $\phi_2,\phi_2$ are non-negative scalar multiples of states in $S(A)^J$; by Theorem \ref{MMextensions} (i), $\phi_2|_J$ and $\psi_2|_J$ are maximally mixed functionals on $J$. Thus, so are their norm-preserving positive extensions to
$J+\C1$ (Theorem \ref{MMextensions} (i)).
Since $J+\C1$ has the Dixmier property (seen in the proof of (ii)), we have by Corollary \ref{DPconvex} (ii) that
\[
D_{J+\C1}((\phi_2+\psi_2)|_{J+\C1}) = D_{J+\C1}(\phi_2|_{J+\C1}) + D_{J+\C1}(\psi_2|_{J+\C1}).
 \]
Further, by Proposition \ref{Dextensions} (i), the same holds restricting $\phi_2$ and $\psi_2$ to $J$:
\[
D_{J}((\phi_2+\psi_2)|_{J}) = D_{J}(\phi_2|_{J}) + D_{J}(\psi_2|_{J}).
\]
Then we have
\begin{align*}
D_A(\phi+\psi)
&= D_A(\phi_1+\phi_2+\psi_1+\psi_2) \\
&= D_A(\phi_1+\psi_1)+D_A(\phi_2+\psi_2) \\
&= D_A(\phi_1)+D_A(\psi_1)+D_J((\phi_2+\psi_2)|_J) \circ \iota_J^* \\
&= D_A(\phi_1)+D_A(\psi_1)+(D_J(\phi_2|_J)+D_J(\psi_2))\circ\iota_J^* \\
&= D_A(\phi_1)+D_A(\psi_1)+D_A(\phi_2)+D_A(\psi_2) \\
&= D_A(\phi)+D_A(\psi)
\end{align*}
where we used Proposition \ref{Dextensions} (i) in the third and fifth equalities, and Proposition \ref{ABstatesConvexity} (the case that one of the functionals is tracial) in the second, third, and final equalities.
\end{proof}

\end{document}